\documentclass[english]{amsart}
\usepackage[T1]{fontenc}
\usepackage[latin9]{inputenc}
\usepackage{amstext}
\usepackage{amsthm}
\usepackage{amssymb}
\usepackage{esint}

\makeatletter
\numberwithin{equation}{section}
\numberwithin{figure}{section}
\theoremstyle{plain}
\newtheorem{thm}{\protect\theoremname}
\theoremstyle{plain}
\newtheorem{lem}[thm]{\protect\lemmaname}
\theoremstyle{plain}
\newtheorem{cor}[thm]{\protect\corollaryname}

\makeatother

\usepackage{babel}
\providecommand{\corollaryname}{Corollary}
\providecommand{\lemmaname}{Lemma}
\providecommand{\theoremname}{Theorem}

\begin{document}
\title{On Certain Genus $0$ Entire Functions}
\author{Ruiming Zhang}
\email{ruimingzhang@guet.edu.cn}
\address{School of Mathematics and Computing Sciences\\
Guilin University of Electronic Technology\\
Guilin, Guangxi 541004, P. R. China. }
\thanks{This work is supported by the National Natural Science Foundation
of China, grant No. 11771355 and No. 12161022.}
\begin{abstract}
In this work we prove that an entire function $f(z)$ has only negative
zeros if and only if its order is strictly less $1$, its root sequence
is real-part dominating and there exists an nonnegative integer $m$
the real function $\left(-\frac{1}{x}\right)^{m}\frac{d^{k}}{dx^{k}}\left(x^{k+m}\frac{d^{m}}{dx^{m}}\left(\frac{f'(x)}{f(x)}\right)\right)$
are completely monotonic on $(0,\infty)$ for all nonnegative integer
$k$. As an application we state a necessary and sufficient condition
for the Riemann hypothesis and generalized Riemann hypothesis for
a primitive Dirichlet character.
\end{abstract}

\maketitle

\section{\label{sec:Intro} Introduction }

In this work we study the zeros of entire function $f(z)$, 

\begin{equation}
f(z)=\sum_{n=0}^{\infty}a_{n}z^{n},\quad a_{0}\cdot a_{n}>0,\label{eq:1.1}
\end{equation}
of order $\rho(f)$ strictly less than $1$ where $\rho(f)$ is defined
by

\begin{equation}
\rho(f)=\limsup_{n\to\infty}\frac{n\log n}{-\log\left|a_{n}\right|}.\label{eq:1.2}
\end{equation}
Many transcendental $q$-special functions such as Ramanujan's entire
function $A_{q}(-z)$ is of this type, \cite{Ismail}. 

Since $\rho(f)<1$, then $f(z)$ is necessarily of genus $0$. If
it has any roots, then there exists a sequence of complex numbers
$\left\{ \lambda_{n}\right\} _{n=1}^{\infty}$ such that \cite{Ahlfors,Boas}
\begin{equation}
\frac{f(z)}{f(0)}=\prod_{n=1}^{\infty}\left(1+\frac{z}{\lambda_{n}}\right)\label{eq:1.3}
\end{equation}
with
\begin{equation}
\sum_{n=1}^{\infty}\frac{1}{\left|\lambda_{n}\right|}<\infty.\label{eq:1.4}
\end{equation}
Furthermore, for any $1>\alpha>\rho(f)$, \cite[Theorem 2.5.18]{Boas}
\begin{equation}
\sum_{n=1}^{\infty}\frac{1}{\left|\lambda_{n}\right|^{\alpha}}<\infty.\label{eq:1.5}
\end{equation}
In this work we prove the following:
\begin{thm}
\label{thm:main} Let $f(z)$ be an entire function defined in (\ref{eq:1.1}),
(\ref{eq:1.3}) such that it has a nonzero root and it's of order
strictly less than $1$. If there is a positive number $\beta_{0}$
with $\beta_{0}\in(0,1)$ such that the sequence $\left\{ \lambda_{n}\right\} _{n\in\mathbb{N}}$
in (\ref{eq:1.3}) and (\ref{eq:1.4}) satisfies 
\begin{equation}
\Re(\lambda_{n})\ge\beta_{0}\left|\lambda_{n}\right|>0,\quad\forall n\in\mathbb{N}.\label{eq:1.6}
\end{equation}
Then $f(z)$ has only negative zeros if and only if there exists an
nonnegative integer $m\in\mathbb{N}_{0}$, all the functions
\begin{equation}
G_{k}^{(m)}(x)=\frac{\left(-1\right)^{m}}{x^{m+k}}\left(x\frac{d}{dx}x\right)^{k}\left(x^{m}\frac{d^{m}}{dx^{m}}\left(\frac{f'(x)}{f(x)}\right)\right),\quad\forall k\in\mathbb{N}_{0}\label{eq:1.7}
\end{equation}
are completely monotonic on $(0,\infty)$. The function $G_{k}^{(m)}(x)$
also has a simpler representation,
\begin{equation}
G_{k}^{(m)}(x)=\left(-\frac{1}{x}\right)^{m}\frac{d^{k}}{dx^{k}}\left(x^{k+m}\frac{d^{m}}{dx^{m}}\left(\frac{f'(x)}{f(x)}\right)\right),\label{eq:1.8}
\end{equation}
 they also satisfy the recursion,
\begin{equation}
\begin{aligned} & G_{0}^{(m)}(z)=(-1)^{m}\frac{d^{m}}{dz^{m}}\left(\frac{f'(z)}{f(z)}\right),\\
 & G_{k}^{(m)}(z)=(k+m)G_{k-1}^{(m)}(z)+z\frac{d}{dz}G_{k-1}^{(m)}(z),\quad\forall k\in\mathbb{N}.
\end{aligned}
\label{eq:1.9}
\end{equation}
\end{thm}

Since the Ramanujan's entire function $A_{q}(z)$ is an entire function
of order $0$ with only positive zeros for $0<q<1$, \cite{Ismail}.
Then for all nonnegative integer $k,m$ and $0<q<1$, the real functions
$\frac{(-1)^{m+1}}{x^{m}}\frac{d^{k}}{dx^{k}}\left(x^{k+m}\frac{d^{m}}{dx^{m}}\left(\frac{A'_{q}(-x)}{A_{q}(-x)}\right)\right)$
are completely monotonic on $(0,\infty)$. However, there are many
other transcendental special functions of the form

\begin{equation}
g(z)=\sum_{n=0}^{\infty}(-1)^{n}a_{n}z^{2n},\quad a_{0}\cdot a_{n}>0,\label{eq:1.10}
\end{equation}
for example, the Riemann $\Xi(z)$, and the Bessel functions $\frac{J_{\nu}(z)}{z^{\nu}}$
, $\frac{J_{\nu}^{(2)}(z;q)}{z^{\nu}}$ with $\nu>-1,\ 0<q<1$, \cite{AndrewsAskeyRoy,Ismail}.
If the order of $g(z)$ is strictly less than $2$,
\begin{equation}
\rho(g)=\limsup_{n\to\infty}\frac{2n\log(2n)}{-\log\left|a_{n}\right|}=2\rho(f)<2,\label{eq:1.11}
\end{equation}
then it is not hard to see that Theorem (\ref{thm:main}) can be applied
to $g(z)$ that has at less one zero. The condition (\ref{eq:1.6})
is replaced by
\begin{equation}
\Re(z_{n}^{2})>0,\quad\left|\Im(z_{n})\right|\le M,\quad\forall n\in\mathbb{N}.\label{eq:1.12}
\end{equation}
Here $M>0$ is a positive constant that only depends on $g(z)$, and
$\left\{ \pm z_{n}\right\} _{n\in\mathbb{N}}$ are the roots of $g(z)$
with $\Re(z_{n})>0,\ \forall n\in\mathbb{N}$. 

Clearly, $f(z)=g(i\sqrt{z})$ transforms $g(z)$ in (\ref{eq:1.10})
into $f(z)$ in (\ref{eq:1.1}) of order strictly less than $1$,
and the zeros $\left\{ \pm z_{n}\right\} _{n\in\mathbb{N}}$ of $g(z)$
into the zeros $\left\{ -\lambda_{n}\right\} _{n\in\mathbb{N}}$ of
$f(z)$ with $\lambda_{n}=z_{n}^{2}$. Therefore, $\lambda_{n}>0$
if and only if $z_{n}\in\mathbb{R}$.

Since $g(z)$ is an entire function not identically $0$, then $\left\{ \pm z_{n}\right\} _{n\in\mathbb{N}}$
has no limit points on the finite complex plane. By (\ref{eq:1.10})
and (\ref{eq:1.12}) we must have
\begin{equation}
\lim_{N\to\infty}\left|z_{n}\right|=+\infty,\quad\lim_{N\to\infty}\Re(z_{n})=+\infty.\label{eq:1.13}
\end{equation}
 Without loss of generality we may let
\begin{equation}
0<\Re(z_{1})\le\Re(z_{2})<\dots\le\Re(z_{n})\le\Re(z_{n+1})\le\dots.\label{eq:1.14}
\end{equation}
Then for any $\epsilon\in(0,1)$ there exists a $N_{\epsilon}\in\mathbb{N}$
such that
\begin{equation}
\epsilon\left|z_{n}\right|^{2}\ge2M^{2},\quad\forall n\ge N_{\epsilon}.\label{eq:1.15}
\end{equation}
Hence, for all $n\ge N_{\epsilon}$
\begin{equation}
0<\Re(z_{n}^{2})=\left|z_{n}\right|^{2}-2\left(\Im(z_{n})\right)^{2}\ge\left|z_{n}\right|^{2}-2M^{2}\ge(1-\epsilon)\left|z_{n}\right|^{2}.\label{eq:1.16}
\end{equation}
Let 
\begin{equation}
\beta_{0}=\min\left\{ 1-\epsilon,\frac{\Re(z_{1}^{2})}{\left|z_{1}^{2}\right|},\dots,\frac{\Re(z_{N_{\epsilon}}^{2})}{\left|z_{N_{\epsilon}}^{2}\right|}\right\} \in(0,1),\label{eq:1.17}
\end{equation}
then it is clear that 
\begin{equation}
\Re(z_{n}^{2})\ge\beta_{0}\left|z_{n}^{2}\right|,\quad\forall n\in\mathbb{N},\label{eq:1.18}
\end{equation}
thus $\lambda_{n}=z_{n}^{2}$ satisfies (\ref{eq:1.6}). 

It is known that for $\nu>-1$ the even entire function $\frac{J_{\nu}(z)}{z^{\nu}}$
is of order $1$ with only real zeros, while for $\nu>-1,\ 0<q<1$
the $q$-Bessel function $\frac{J_{\nu}^{(2)}(z;q)}{z^{\nu}}$ is
an even entire functions that of order $0$ with only real zeros,
\cite{Ismail}. Applying Theorem \ref{thm:main} to the Bessel functions
we conclude that for all nonnegative integers $k,m$ and any positive
number $q$ satisfying $0<q<1$, the functions $\left(-\frac{1}{x}\right)^{m}\frac{d^{k}}{dx^{k}}\left(x^{k+m}\frac{d^{m}}{dx^{m}}\left(\frac{f'(x)}{f(x)}\right)\right)$
are completely monotonic on $(0,\infty)$ where $f(x)=x^{-\nu/2}I_{\nu}(\sqrt{x})$
or $f(x)=x^{-\nu/2}I_{\nu}^{(2)}(\sqrt{x};q)$.

\section{\label{sec:proof}Proof of the Theorem }

\subsection{The heat kernel $\Theta(x)$}
\begin{lem}
\label{lem:2} Let $f(z)$ be an entire function defined in (\ref{eq:1.1}),
(\ref{eq:1.3}) with order $\rho(f)<1$ such that there exists a positive
number $\beta_{0}\in(0,1)$ it satisfies (\ref{eq:1.6}). Then its
associated heat kernel
\begin{equation}
\Theta(t)=\sum_{n=1}^{\infty}e^{-\lambda_{n}t},\quad t>0\label{eq:2.1}
\end{equation}
 is in $C^{\infty}(0,\infty)$, and for any 
\begin{equation}
k\in\mathbb{N}_{0},\ 1>\alpha>p(f),\ 0<\beta<\beta_{0}\inf\left\{ \left|\lambda_{n}\right|\bigg|n\in\mathbb{N}\right\} ,\label{eq:2.2}
\end{equation}
the function $\Theta(t)$ satisfies 

\begin{equation}
\Theta^{(k)}(t)=\mathcal{O}\left(t^{-\alpha-k}\right),\quad t\to0^{+}\label{eq:2.3}
\end{equation}
and 
\begin{equation}
\Theta^{(k)}(t)=\mathcal{O}\left(e^{-\beta t}\right),\quad t\to+\infty.\label{eq:2.4}
\end{equation}
Furthermore, for all $k\in\mathbb{N}_{0}$and $x\ge0$,
\begin{equation}
\int_{0}^{\infty}t^{k}e^{-xt}\left|\Theta(t)\right|dt\le\frac{k!}{\beta_{0}^{k+1}}\sum_{n=1}^{\infty}\frac{1}{\left|\lambda_{n}\right|^{k+1}}<\infty\label{eq:2.5}
\end{equation}
and
\begin{equation}
(-1)^{k}\left(\frac{f'(x)}{f(x)}\right)^{(k)}=\sum_{n=1}^{\infty}\frac{k!}{(x+\lambda_{n})^{k+1}}=\int_{0}^{\infty}e^{-xt}t^{k}\Theta(t)dt.\label{eq:2.6}
\end{equation}
\end{lem}

\begin{proof}
Since for any $y>0$
\begin{equation}
\sup_{x>0}x^{y}e^{-x}=\left(\frac{y}{e}\right)^{y},\label{eq:2.7}
\end{equation}
 then for any $t>0,\,k\ge0$ by (\ref{eq:1.5}) and (\ref{eq:1.6}),
\begin{equation}
\begin{aligned} & \sum_{n=1}^{\infty}\left|\lambda_{n}^{k}e^{-\lambda_{n}t}\right|=\sum_{n=1}^{\infty}\left|\lambda_{n}\right|^{k}e^{-\Re(\lambda_{n})t}\le\sum_{n=1}^{\infty}\frac{\left(\beta_{0}t\left|\lambda_{n}\right|\right)^{k+\alpha}e^{-\beta_{0}\left|\lambda_{n}\right|t}}{(\beta_{0}t)^{\alpha+k}\left|\lambda_{n}\right|^{\alpha}}\\
 & \le\frac{\sup_{x>0}x^{k+\alpha}e^{-x}}{(\beta_{0}t)^{\alpha+k}}\sum_{n=1}^{\infty}\frac{1}{\left|\lambda_{n}\right|^{\alpha}}=\left(\frac{k+\alpha}{e\beta_{0}t}\right)^{k+\alpha}\sum_{n=1}^{\infty}\frac{1}{\left|\lambda_{n}\right|^{\alpha}}<\infty,
\end{aligned}
\label{eq:2.8}
\end{equation}
which proves that $\Theta(t)\in C^{\infty}(0,\infty)$ and (\ref{eq:2.3}). 

Let $0<\beta<\beta_{0}\inf\left\{ \left|\lambda_{n}\right|\bigg|n\in\mathbb{N}\right\} $,
then there exists a positive number $\epsilon$ with $0<\epsilon<1$
such that 
\begin{equation}
\beta=\epsilon\beta_{0}\inf\left\{ \left|\lambda_{n}\right|\bigg|n\in\mathbb{N}\right\} \le\epsilon\beta_{0}\left|\lambda_{n}\right|,\quad\forall n\in\mathbb{N}.\label{eq:2.9}
\end{equation}
Thus for $t\ge1$ by (\ref{eq:2.8}),
\begin{equation}
\begin{aligned} & \left|e^{\beta t}\Theta^{(k)}(t)\right|\le\sum_{n=1}^{\infty}\left|\lambda_{n}\right|^{k}e^{-(\beta_{0}\left|\lambda_{n}\right|-\beta)t}\le\sum_{n=1}^{\infty}\left|\lambda_{n}\right|^{k}e^{-(1-\epsilon)\beta_{0}\left|\lambda_{n}\right|t}\\
 & \le\sum_{n=1}^{\infty}\left|\lambda_{n}\right|^{k}e^{-(1-\epsilon)\beta_{0}\left|\lambda_{n}\right|}<\infty,
\end{aligned}
\label{eq:2.10}
\end{equation}
which establishes (\ref{eq:2.4}).

For all $x,k\ge0$ since
\begin{equation}
\begin{aligned} & \int_{0}^{\infty}t^{k}e^{-xt}\left|\Theta(t)\right|dt\le\int_{0}^{\infty}t^{k}\left(\sum_{n=1}^{\infty}e^{-(x+\Re(\lambda_{n}))t}\right)dt\le\sum_{n=1}^{\infty}\int_{0}^{\infty}t^{k}e^{-(x+\beta_{0}\left|\lambda_{n}\right|)t}dt\\
 & \le\sum_{n=1}^{\infty}\frac{k!}{\left(x+\beta_{0}\left|\lambda_{n}\right|\right)^{k+1}}\le\frac{k!}{\beta_{0}^{k+1}}\sum_{n=1}^{\infty}\frac{1}{\left|\lambda_{n}\right|^{k+1}}<\infty,
\end{aligned}
\label{eq:2.11}
\end{equation}
then,
\begin{equation}
\begin{aligned} & \left(\frac{f'(x)}{f(x)}\right)^{(k)}=\sum_{n=1}^{\infty}\frac{(-1)^{k}k!}{(x+\lambda_{n})^{k+1}}\\
 & =(-1)^{k}\sum_{n=1}^{\infty}\int_{0}^{\infty}t^{k}e^{-(x+\lambda_{n})t}dt=(-1)^{k}\int_{0}^{\infty}e^{-xt}t^{k}\Theta(t)dt.
\end{aligned}
\label{eq:2.12}
\end{equation}
\end{proof}
\begin{lem}
\label{lem:3}Let $f(z)$, $\Theta(x)$ and $\left\{ \lambda_{n}\right\} _{n=1}^{\infty}$
be defined as in Theorem \ref{thm:main} and Lemma \ref{lem:2}. Then
the sequence $\left\{ \lambda_{n}\right\} _{n=1}^{\infty}$ is positive
if and only if the function $\Theta(x)$ is completely monotonic on
$(0,\infty)$. 
\end{lem}

\begin{proof}
Assume that $\left\{ \lambda_{n}\right\} _{n=1}^{\infty}$ is positive,
then 
\begin{equation}
(-1)^{k}\Theta^{(k)}(t)=\sum_{n=1}^{\infty}\lambda_{n}^{k}e^{-\lambda_{n}t}\ge0,\quad\forall t>0,k\in\mathbb{N}_{0},\label{eq:2.13}
\end{equation}
 hence $\Theta(t)$ is completely monotonic by definition.

On the other hand if $\Theta(t)$ is completely monotonic on $(0,\infty)$,
then by the Bernstein's theorem there exists a nonnegative measure
$d\mu(t)$ on $[0,\infty)$ such that \cite{SchillingSongVondracek,WidderBook}
\begin{equation}
\Theta(t)=\sum_{n=1}^{\infty}e^{-\lambda_{n}t}=\int_{0}^{\infty}e^{-yt}d\mu(y),\quad\forall t>0.\label{eq:2.14}
\end{equation}
For any $x\ge0$ we multiply $e^{-xt}$ both sides of (\ref{eq:2.14})
and integrate it with respect to $t$ from $0$ to $\infty$ to get
\begin{equation}
\begin{aligned} & \int_{0}^{\infty}e^{-xt}\Theta(t)dt=\sum_{n=1}^{\infty}\frac{1}{x+\lambda_{n}}=\frac{f'(x)}{f(x)}\\
 & =\int_{0}^{\infty}e^{-xt}\left(\int_{0}^{\infty}e^{-yt}d\mu(y)\right)dx=\int_{0}^{\infty}\frac{d\mu(y)}{x+y}.
\end{aligned}
\label{eq:2.15}
\end{equation}
The exchange of orders between summation and integration, and between
integrations are guaranteed by (\ref{eq:2.5}), (\ref{eq:2.6}) and
Fubini's theorem.

By Fatou's Lemma and (\ref{eq:2.5}), (\ref{eq:2.6}) and (\ref{eq:2.15}),
\begin{equation}
\begin{aligned} & 0<\int_{0}^{\infty}\frac{d\mu(y)}{y}=\int_{0}^{\infty}\liminf_{x\downarrow0}\frac{d\mu(y)}{x+y}\le\liminf_{x\downarrow0}\int_{0}^{\infty}\frac{d\mu(y)}{x+y}\\
 & =\liminf_{x\downarrow0}\left|\frac{f'(x)}{f(x)}\right|\le\liminf_{x\downarrow0}\left|\sum_{n=1}^{\infty}\frac{1}{x+\Re(\lambda_{n})}\right|\le\frac{1}{\beta_{0}}\sum_{n=1}^{\infty}\frac{1}{\left|\lambda_{n}\right|}.
\end{aligned}
\label{eq:2.16}
\end{equation}
For $x,y>0$ since $\frac{1}{(x+y)^{2}}<\frac{1}{4xy}$ then,
\begin{equation}
\int_{0}^{\infty}\frac{d\mu(y)}{(x+y)^{2}}\le\frac{1}{4x}\int_{0}^{\infty}\frac{d\mu(y)}{y}<\infty.\label{eq:2.17}
\end{equation}
For $x,y>0,\ \left|h\right|<\frac{x}{2}$ by (\ref{eq:2.15}), Lebesgue's
dominated convergence theorem and the inequality
\begin{equation}
\frac{2}{\left|(x+y)(x+y+h)\right|}\le\frac{1}{(x+y)^{2}}+\frac{1}{(x/2+y)^{2}}\label{eq:2.18}
\end{equation}
we obtain
\begin{equation}
\begin{aligned} & -\frac{d}{dx}\left(\frac{f'(x)}{f(x)}\right)=-\frac{d}{dx}\int_{0}^{\infty}\frac{d\mu(y)}{x+y}\\
 & =\lim_{h\to0}\frac{1}{h}\int_{0}^{\infty}\left(\frac{1}{x+y}-\frac{1}{x+y+h}\right)d\mu(y)\\
 & =\lim_{h\to0}\int_{0}^{\infty}\frac{d\mu(y)}{(x+y)(x+y+h)}=\int_{0}^{\infty}\frac{d\mu(y)}{(x+y)^{2}}.
\end{aligned}
\label{eq:2.19}
\end{equation}
Then by Fatou's Lemma and (\ref{eq:2.6}),
\begin{equation}
\begin{aligned} & 0<\int_{0}^{\infty}\frac{d\mu(y)}{y^{2}}=\int_{0}^{\infty}\liminf_{x\downarrow0}\frac{d\mu(y)}{(x+y)^{2}}\le\liminf_{x\downarrow0}\int_{0}^{\infty}\frac{d\mu(y)}{(x+y)^{2}}\\
 & =\liminf_{x\downarrow0}\left|\frac{d}{dx}\left(\frac{f'(x)}{f(x)}\right)\right|\le\liminf_{x\downarrow0}\left|\sum_{n=1}^{\infty}\frac{1}{(x+\Re(\lambda_{n}))^{2}}\right|\le\frac{1}{\beta_{0}^{2}}\sum_{n=1}^{\infty}\frac{1}{\left|\lambda_{n}\right|^{2}}.
\end{aligned}
\label{eq:2.20}
\end{equation}

Let 
\begin{equation}
s(z)=\int_{0}^{\infty}\frac{d\mu(y)}{z+y},\label{eq:2.21}
\end{equation}
we first show that $s(z)$ is continuous and bounded on $\Re(z)\ge0$
and analytic in $\Re(z)>0$. 

For $\Re(z)\ge0$ since 
\begin{equation}
\int_{0}^{\infty}\frac{d\mu(y)}{\left|z+y\right|}\le\int_{0}^{\infty}\frac{d\mu(y)}{\Re(z)+y}\le\int_{0}^{\infty}\frac{d\mu(y)}{y}<\infty\label{eq:2.22}
\end{equation}
and
\begin{equation}
\int_{0}^{\infty}\frac{d\mu(y)}{\left|z+y\right|^{2}}\le\int_{0}^{\infty}\frac{d\mu(y)}{(\Re(z)+y)^{2}}\le\int_{0}^{\infty}\frac{d\mu(y)}{y^{2}}<\infty,\label{eq:2.23}
\end{equation}
then both integrals
\begin{equation}
\int_{0}^{\infty}\frac{d\mu(y)}{z+y},\quad\int_{0}^{\infty}\frac{d\mu(y)}{(z+y)^{2}}\label{eq:2.24}
\end{equation}
converge absolutely for $\Re(z)\ge0$ and uniformly on any of its
compact subsets. Thus they are continuous on $\Re(z)\ge0$ by (\ref{eq:2.22})
and (\ref{eq:2.23}). Similar to the proof of (\ref{eq:2.19}) we
can show that 
\begin{equation}
s'(z)=-\int_{0}^{\infty}\frac{d\mu(y)}{(z+y)^{2}},\quad\text{\ensuremath{\Re(z)>0,} }\label{eq:2.25}
\end{equation}
 Then $s(z)$ is analytic in $\Re(z)>0$.

Next we prove that $s(z)$ is continuous on any of its compact subsets
of $\left\{ z\vert\Re(z)\le0,\,\Im(z)\neq0\right\} $ and analytic
in $\left\{ z\vert\Re(z)<0,\,\Im(z)\neq0\right\} $. 

Let $\Re(z)=-a\le0$ and $b=\Im(z)\neq0$, then for any positive number
$\epsilon>0$,
\begin{equation}
\begin{aligned} & \int_{0}^{\infty}\frac{d\mu(y)}{\left|(y-a)+bi\right|}=\int_{0}^{a+\epsilon}\frac{d\mu(y)}{\left|(y-a)+bi\right|}+\int_{a+\epsilon}^{\infty}\frac{d\mu(y)}{\left|(y-a)+bi\right|}\\
 & \le\frac{1}{\left|b\right|}\int_{0}^{a+\epsilon}\frac{yd\mu(y)}{y}+\int_{a+\epsilon}^{\infty}\frac{y}{(y-a)}\frac{d\mu(y)}{y}\\
 & \le\frac{a+\epsilon}{\left|b\right|}\int_{0}^{a+\epsilon}\frac{d\mu(y)}{y}+\int_{a+\epsilon}^{\infty}\frac{y-a+a}{y-a}\frac{d\mu(y)}{y}\\
 & \le\frac{a+\epsilon}{\left|b\right|}\int_{0}^{a+\epsilon}\frac{d\mu(y)}{y}+\left(1+\frac{a}{\epsilon}\right)\int_{a+\epsilon}^{\infty}\frac{d\mu(y)}{y}\\
 & \le\left(1+\frac{a}{\epsilon}+\frac{a+\epsilon}{\left|b\right|}\right)\int_{0}^{\infty}\frac{d\mu(y)}{y}<\infty
\end{aligned}
\label{eq:2.26}
\end{equation}
and
\begin{equation}
\begin{aligned} & \int_{0}^{\infty}\frac{d\mu(y)}{\left|(y-a)+bi\right|^{2}}\le\frac{1}{b^{2}}\int_{0}^{a+\epsilon}\frac{y^{2}d\mu(y)}{y^{2}}+\int_{a+\epsilon}^{\infty}\frac{(y-a+a)^{2}}{(y-a)^{2}}\frac{d\mu(y)}{y^{2}}\\
 & \le\left(\frac{a+\epsilon}{b}\right)^{2}\int_{0}^{a+\epsilon}\frac{d\mu(y)}{y^{2}}+\int_{a+\epsilon}^{\infty}\left(1+\frac{a}{y-a}\right)^{2}\frac{d\mu(y)}{y^{2}}\\
 & \le\left(\frac{a+\epsilon}{b}\right)^{2}\int_{0}^{a+\epsilon}\frac{d\mu(y)}{y^{2}}+\left(1+\frac{a}{\epsilon}\right)^{2}\int_{a+\epsilon}^{\infty}\frac{d\mu(y)}{y^{2}}\\
 & \le\left(\left(\frac{a+\epsilon}{b}\right)^{2}+\left(1+\frac{a}{\epsilon}\right)^{2}\right)\int_{0}^{\infty}\frac{d\mu(y)}{y^{2}}<\infty.
\end{aligned}
\label{eq:2.27}
\end{equation}
Then both integrals 
\begin{equation}
\int_{0}^{\infty}\frac{d\mu(y)}{y+z},\quad\int_{0}^{\infty}\frac{d\mu(y)}{\left(y+z\right)^{2}}\label{eq:2.28}
\end{equation}
converge absolutely and uniformly on any of the compact subsets of
$\left\{ z\vert\Re(z)\le0,\,\Im(z)\neq0\right\} $. Similar to the
case $\Re(z)\ge0$, the above estimates imply that $s(z)$ is continuous
on $\left\{ z\vert\Re(z)\le0,\,\Im(z)\neq0\right\} $, analytic in
$\left\{ z\vert\Re(z)<0,\,\Im(z)\neq0\right\} $.

For any $z_{0}$ with $\Re(z_{0})=0$ and $\Im(z_{0})\neq0$, there
exists a positive number $\epsilon>0$ such that the disk $D(z_{0},\epsilon)=\left\{ z:\left|z-z_{0}\right|<\epsilon\right\} $
is completely contained in $\left\{ z:z\in\mathbb{C}\backslash\left(-\infty,0\right],\,\Im(z)>0\right\} $
or $\left\{ z:z\in\mathbb{C}\backslash\left(-\infty,0\right],\,\Im(z)<0\right\} $.
Without losing any generality we may assume that $\Im(z_{0})>0$ and
$D(z_{0},\epsilon)\subset\left\{ z:z\in\mathbb{C}\backslash\left(-\infty,0\right],\,\Im(z)>0\right\} $,
and the other case can be treated similarly.

We first observe that $s(z)$ is continuous on the closure of $D(z_{0},\epsilon/3)$
except possibly on its intersection with the imaginary axis $\Re(z)=0$
where it has only one-sided continuity. For any $z$ in the intersection,
let $\left\{ z_{j}\right\} $ be a sequence in the closure of $D(z_{0},\epsilon/3)$
that converges to $z$. If all $\left\{ z_{j}\right\} $ except finitely
many are completely inside one of $\Re(z)\ge0$ and $\left\{ z\vert\Re(z)\le0,\,\Im(z)\ge\epsilon/3\right\} $,
then by the one-sided continuity we have $s(z_{j})\to s(z)$ as $j\to\infty$.
Otherwise, the sequence can be decomposed into two converging subsequences
$\left\{ z_{j_{1}}\right\} $ and $\left\{ z_{j_{2}}\right\} $, which
are completely contained inside in $\Re(z)\ge0$ and $\left\{ z\vert\Re(z)\le0,\,\Im(z)\ge\epsilon/3\right\} $
respectively. However, by the one-sided continuities proved on both
sets we have $s(z_{j_{k}})\to s(z)$ as $j_{k}\to\infty$ for $k=1,2$.
Therefore, $s(z)$ is continuous on the closure of $D(z_{0},\epsilon/3)$.

Given any closed piecewise $C^{1}$ curve $\gamma$ in $D(z_{0},\epsilon/3)$.
If it does not intersect the imaginary axis $\Re(z)=0$, then it's
completely contained in one of two disjoint regions, $\Re(z)>0$ and
$\left\{ z\vert\Re(z)<0,\,\Im(z)>0\right\} $. From the above discussion
$s(z)$ is analytic in both of them, then by Cauchy's integration
theorem we have $\oint_{\gamma}s(z)dz=0$. Otherwise, $\gamma$ can
be decomposed into a sum of closed piecewise $C^{1}$ curves $\left\{ \gamma_{j}\right\} $
in $D(z_{0},\epsilon/3)$ by adding the line segments with opposed
directions on $\Re(z)=0$ enclosed by $\gamma$. Then $s(z)$ is continuous
on the closed sets enclosed by each $\gamma_{j}$ and analytic in
their interiors. By a weaker version of Cauchy's integration theorem
for each $j$ we have $\oint_{\gamma_{j}}s(z)dz=0$ and \cite{Ahlfors}
\begin{equation}
\oint_{\gamma}s(z)dz=\sum_{j}\oint_{\gamma_{j}}s(z)dz=\sum_{j}0=0.\label{eq:2.29}
\end{equation}
Then by Morera's theorem $s(z)$ is analytic in $D(z_{0},\epsilon/3)$,
in particular at $z_{0}$. 

In the summary of above discussion we have proved that $s(z)$ is
continuous on $\mathbb{C}\backslash\left(-\infty,0\right)$ and analytic
in $\mathbb{C}\backslash\left(-\infty,0\right]$. 

Since the meromorphic function
\begin{equation}
\frac{f'(z)}{f(z)}=\sum_{n=1}^{\infty}\frac{1}{z+\lambda_{n}}\label{eq:2.30}
\end{equation}
is analytic in $\mathbb{C}\backslash\left\{ -\lambda_{n}\right\} _{n\in\mathbb{N}}$,
and $s(z)$ is analytic in $\mathbb{C}\backslash\left(-\infty,0\right]$,
while by (\ref{eq:2.15}) they are equal on $z\ge0$, then we must
have
\begin{equation}
\frac{f'(z)}{f(z)}=s(z),\quad z\in\mathbb{C}\backslash\left(\left\{ -\lambda_{n}\right\} _{n\in\mathbb{N}}\cup\left(-\infty,0\right]\right)\label{eq:2.31}
\end{equation}
 by analytic continuation.

If there exists an $n_{0}\in\mathbb{N}$ such that $\lambda_{n_{0}}$
is not positive. Since $\Re(\lambda_{n_{0}})>0$ and $-\lambda_{n_{0}}$
is not negative, then $s(z)$ is analytic at $-\lambda_{n_{0}}$.
Hence, 
\begin{equation}
\lim_{z\to-\lambda_{n_{0}}}s(z)=s\left(-\lambda_{n_{0}}\right)=\int_{0}^{\infty}\frac{d\mu(y)}{y-\lambda_{n_{0}}}\neq\infty.\label{eq:2.32}
\end{equation}
On the other hand since $-\lambda_{n_{0}}$ is a zero of $f(z)$,
then it is a simple pole of $\frac{f'(z)}{f(z)}$ with a positive
residue. Hence, 
\begin{equation}
\lim_{z\to-\lambda_{n_{0}}}\frac{f'(z)}{f(z)}=\infty.\label{eq:2.33}
\end{equation}
Combining (\ref{eq:2.32}) and (\ref{eq:2.33}) we have reached the
following contradiction,
\begin{equation}
\infty\neq s\left(-\lambda_{n_{0}}\right)=\lim_{z\to-\lambda_{n_{0}}}s(z)=\lim_{z\to-\lambda_{n_{0}}}\frac{f'(z)}{f(z)}=\infty.\label{eq:2.34}
\end{equation}
This contradiction proves that $\left\{ \lambda_{n}\right\} _{n\in\mathbb{N}}$
must be a positive sequence.
\end{proof}

\subsection{Proof of Theorem }
\begin{proof}
For $k,m\in\mathbb{N}_{0},\,t>0$ and $z$ with $\Re(z)>0$ we define
\begin{equation}
\Theta_{k}(t)=(-t)^{k}\Theta^{(k)}(t),\quad G_{k}^{(m)}(z)=\int_{0}^{\infty}e^{-zt}t^{m}\Theta_{k}(t)dt,\quad\forall k\in\mathbb{N}\label{eq:2.35}
\end{equation}
and
\begin{equation}
\Theta_{0}(t)=\Theta(t),\quad G_{0}^{(m)}(z)=\int_{0}^{\infty}e^{-zt}t^{m}\Theta(t)dt=(-1)^{m}\frac{d^{m}}{dz^{m}}\left(\frac{f'(z)}{f(z)}\right).\label{eq:2.36}
\end{equation}
For any $k\in\mathbb{N}_{0}$ and $t>0$, it is clear that $(-1)^{k}\Theta^{(k)}(t)\ge0$
if and only if $\Theta_{k}(t)\ge0$. By (\ref{eq:2.3}), (\ref{eq:2.4})
and (\ref{eq:2.6}) $G_{k}^{(m)}(x)\in C^{\infty}(0,\infty)$ exists
for all nonnegative integers $k,m$. 

For each $k,m\in\mathbb{N}_{0}$, apply the Bernstein theorem to $G_{k}^{(m)}(x)$
we get that $\Theta_{k}(t)\ge0$ if and only if $G_{k}^{(m)}(x)$
is completely monotonic on $(0,\infty)$.

For any $k\in\mathbb{N}$ and $x>0$, by (\ref{eq:2.3}), (\ref{eq:2.4})
and integration by parts, 
\begin{equation}
\begin{aligned} & G_{k}^{(m)}(x)=\int_{0}^{\infty}e^{-xt}t^{m}\Theta_{k}(t)dt=\left(-1\right)^{k}\int_{0}^{\infty}t^{k+m}e^{-xt}\frac{d}{dt}\Theta^{(k-1)}(t)dt\\
 & =\int_{0}^{\infty}\left(k+m-xt\right)e^{-xt}t^{m}\Theta_{k-1}(t)dt=(k+m)G_{k-1}^{(m)}(x)+x\frac{d}{dx}G_{k-1}^{(m)}(x)\\
 & =x^{1-k-m}\frac{d}{dx}\left(x^{k+m}G_{k-1}^{(m)}(x)\right)=\frac{1}{x^{k+m}}\left(x\frac{d}{dx}x\right)\left(x^{k-1+m}G_{k-1}^{(m)}(x)\right),
\end{aligned}
\label{eq:2.37}
\end{equation}
then,
\begin{equation}
x^{k+m}G_{k}^{(m)}(x)=\left(x\frac{d}{dx}x\right)\left(x^{k-1+m}G_{k-1}^{(m)}(x)\right)=\dots=\left(x\frac{d}{dx}x\right)^{k}\left((-x)^{m}\frac{d^{m}}{dx^{m}}\left(\frac{f'(x)}{f(x)}\right)\right),\label{eq:2.38}
\end{equation}
it leads to 
\begin{equation}
G_{k}^{(m)}(x)=\frac{\left(-1\right)^{m}}{x^{m+k}}\left(x\frac{d}{dx}x\right)^{k}\left(x^{m}\frac{d^{m}}{dx^{m}}\left(\frac{f'(x)}{f(x)}\right)\right).\label{eq:2.39}
\end{equation}
 On the other hand, for any $k\in\mathbb{N}$ and $x>0$ by (\ref{eq:2.3}),
(\ref{eq:2.4}) and integration by parts, 
\begin{equation}
\begin{aligned} & G_{k}^{(m)}(x)=\int_{0}^{\infty}e^{-xt}t^{m}\Theta_{k}(t)dt=\int_{0}^{\infty}e^{-xt}(-t)^{k}t^{m}\Theta^{(k)}(t)dt\\
 & =\int_{0}^{\infty}\frac{d^{k}}{dt^{k}}\left(e^{-xt}t^{k+m}\right)\Theta(t)dt=\int_{0}^{\infty}\left(e^{y}y^{-m}\frac{d^{k}}{dy^{k}}\left(e^{-y}y^{k+m}\right)\right)\bigg|_{y=xt}e^{-xt}t^{m}\Theta(t)dt\\
 & =k!\int_{0}^{\infty}\left(\frac{e^{y}y^{-m}}{k!}\frac{d^{k}}{dy^{k}}\left(e^{-y}y^{k+m}\right)\right)\bigg|_{y=xt}e^{-xt}t^{m}\Theta(t)dt=k!\int_{0}^{\infty}L_{k}^{(m)}(xt)e^{-xt}t^{m}\Theta(t)dt,
\end{aligned}
\label{eq:2.40}
\end{equation}
which gives
\begin{equation}
G_{k}^{(m)}(x)=k!\int_{0}^{\infty}L_{k}^{(m)}(xt)e^{-xt}t^{m}\Theta(t)dt,\label{eq:2.41}
\end{equation}
where we have applied the Rodrigues formula for Laguerre polynomials,\cite{AndrewsAskeyRoy,Ismail}

\begin{equation}
{\displaystyle L_{n}^{(\alpha)}(x)=\frac{x^{-\alpha}e^{x}}{n!}\frac{d^{n}}{dx^{n}}\left(e^{-x}x^{n+\alpha}\right).}\label{eq:2.42}
\end{equation}
Since
\begin{equation}
L_{k}^{(m)}(xt)=\frac{e^{xt}}{k!x^{m}}\frac{d^{k}}{dx^{k}}\left(e^{-tx}x{}^{k+m}\right),\label{eq:2.43}
\end{equation}
then
\begin{equation}
\begin{aligned} & G_{k}^{(m)}(x)=k!\int_{0}^{\infty}L_{k}^{(m)}(xt)e^{-xt}t^{m}\Theta(t)dt=x^{-m}\int_{0}^{\infty}\frac{d^{k}}{dx^{k}}\left(e^{-tx}x{}^{k+m}\right)t^{m}\Theta(t)dt\\
 & =x^{-m}\frac{d^{k}}{dx^{k}}x^{k+m}\int_{0}^{\infty}e^{-tx}t^{m}\Theta(t)dt=\left(-\frac{1}{x}\right)^{m}\frac{d^{k}}{dx^{k}}\left(x^{k+m}\frac{d^{m}}{dx^{m}}\left(\frac{f'(x)}{f(x)}\right)\right).
\end{aligned}
\label{eq:2.44}
\end{equation}
\end{proof}

\section{An Application}

The Riemann $\xi$-function is defined by \cite{AndrewsAskeyRoy,Edwards},
\begin{equation}
\xi(s)=\pi^{-s/2}(s-1)\Gamma\left(1+\frac{s}{2}\right)\zeta(s),\quad s=\sigma+it,\ \sigma,t\in\mathbb{R},\label{eq:3.1}
\end{equation}
where $\Gamma(s)$, $\zeta(s)$ are the respective analytic continuations
of
\begin{equation}
\Gamma(s)=\int_{0}^{\infty}e^{-x}x^{s-1}dx,\quad\sigma>0\label{eq:3.2}
\end{equation}
and 
\begin{equation}
\zeta(s)=\sum_{n=1}^{\infty}\frac{1}{n^{s}},\quad\sigma>1.\label{eq:3.3}
\end{equation}
It is well-known that $\xi(s)$ is an order $1$ entire function that
satisfies $\xi(s)=\xi(1-s)$, and it has infinitely many zeros, all
of them are in the vertical strip $\sigma\in(0,1)$. Then the Riemann
Xi function $\Xi(s)=\xi\left(\frac{1}{2}+is\right)$ is an even entire
function of order $1$, and it has infinitely many zeros, all of them
are in the horizontal strip $t\in(-1/2,1/2)$. The Riemann hypothesis
is that all the zeros of $\xi(s)$ are on $\sigma=\frac{1}{2},$or
equivalently, all the zeros $\left\{ \pm z_{n}\right\} _{n=1}^{\infty}$
of $\Xi(s)$ are real where $z_{1}\approx14.1347$.

Let
\begin{equation}
\lambda_{n}=z_{n}^{2},\quad n\in\mathbb{N},\label{eq:3.4}
\end{equation}
 then it is know that for any $\epsilon>0$, \cite{Davenport,Edwards}
\begin{equation}
\sum_{n=1}^{\infty}\frac{1}{\left|\lambda_{n}\right|^{1/2+\epsilon}}<\infty.\label{eq:3.5}
\end{equation}
Let 
\begin{equation}
\Theta_{\Xi}(t)=\sum_{n=1}^{\infty}e^{-\lambda_{n}^{2}t},\quad\forall t>0,\label{eq:3.6}
\end{equation}
then for all $k\in\mathbb{N}_{0}$ by Lemma \ref{lem:2} we have
\begin{equation}
\Theta_{\Xi}^{(k)}(t)=\mathcal{O}\left(t^{-1/2-\epsilon-k}\right),\quad t\downarrow0.\label{eq:3.7}
\end{equation}
 Since \cite{Edwards}
\begin{equation}
\Xi(s)=\int_{-\infty}^{\infty}\Phi(u)e^{ius}du=2\int_{0}^{\infty}\Phi(u)\cos(us)du,\label{eq:3.8}
\end{equation}
where 
\begin{equation}
\Phi(u)=\Phi(-u)=\sum_{n=1}^{\infty}\left(4n^{4}\pi^{2}e^{9u/2}-6n^{2}\pi e^{5u/2}\right)e^{-n^{2}\pi e^{2u}}>0.\label{eq:3.9}
\end{equation}
Then,
\begin{equation}
\xi\left(\frac{1}{2}+s\right)=\sum_{n=0}^{\infty}a_{n}s^{2n},\quad a_{n}=\frac{2}{(2n)!}\int_{0}^{\infty}\Phi(u)u^{2n}du>0\label{eq:3.10}
\end{equation}
and
\begin{equation}
f(s)=\xi\left(\frac{1}{2}+\sqrt{s}\right)=\sum_{n=0}^{\infty}a_{n}s^{n}\label{eq:3.11}
\end{equation}
defines an entire function of order $\frac{1}{2}$, and $\Xi(s)$
has only real roots if and only if $f(s)$ has only negative roots.
Apply Theorem \ref{thm:main} to $f(s)$ we obtain the following corollary:
\begin{cor}
\label{cor:5}The Riemann hypothesis is true if and only if there
exists an nonnegative integer $m$, all the functions
\begin{equation}
\left(-\frac{1}{x}\right)^{m}\frac{d^{k}}{dx^{k}}\left(x^{k+m}\frac{d^{m}}{dx^{m}}\left(\frac{\xi'\left(\frac{1}{2}+\sqrt{x}\right)}{\sqrt{x}\xi\left(\frac{1}{2}+\sqrt{x}\right)}\right)\right),\quad k\in\mathbb{N}_{0}\label{eq:3.12}
\end{equation}
are completely monotonic on $(0,\infty)$.
\end{cor}

For a primitive Dirichlet character $\chi(n)$ modulo $q$, let \cite{AndrewsAskeyRoy,Apostol,Davenport,Edwards}
\begin{equation}
\xi(s,\chi)=\left(\frac{q}{\pi}\right)^{(s+\kappa)/2}\Gamma\left(\frac{s+\kappa}{2}\right)L(s,\chi),\label{eq:3.13}
\end{equation}
where $\kappa$ is the parity of $\chi$ and $L(s,\chi)$ is the analytic
continuation of 
\begin{equation}
L\left(s,\chi\right)=\sum_{n=1}^{\infty}\frac{\chi(n)}{n^{s}},\quad\sigma>1.\label{eq:3.14}
\end{equation}
Then $\xi(s,\chi)$ is an entire function of order $1$ such that
\cite{Davenport}
\begin{equation}
\xi(s,\chi)=\epsilon(\chi)\xi(1-s,\overline{\chi}),\label{eq:3.15}
\end{equation}
where 
\begin{equation}
\epsilon(\chi)=\frac{\tau(\chi)}{i^{\kappa}\sqrt{q}},\quad\tau(\chi)=\sum_{n=1}^{q}\chi(n)\exp\left(\frac{2\pi in}{q}\right).\label{eq:3.16}
\end{equation}
Let 
\begin{equation}
G(s,\chi)=\xi\left(s,\chi\right)\cdot\xi\left(s,\overline{\chi}\right),\label{eq:3.17}
\end{equation}
then
\begin{equation}
G(s,\chi)=\epsilon\left(\chi\right)\cdot\epsilon\left(\overline{\chi}\right)G(1-s,\chi).\label{eq:3.18}
\end{equation}
Since \cite{Davenport}
\begin{equation}
\tau\left(\overline{\chi}\right)=\overline{\tau(\chi)},\quad\left|\tau(\chi)\right|=\sqrt{q},\label{eq:3.19}
\end{equation}
then
\begin{equation}
G(s,\chi)=G(1-s,\chi).\label{eq:3.20}
\end{equation}
Since the entire function $\xi\left(\frac{1}{2}+is,\chi\right)$ has
an integral representation, \cite{Davenport}
\begin{equation}
\xi\left(\frac{1}{2}+is,\chi\right)=\int_{-\infty}^{\infty}e^{isy}\varphi\left(y,\chi\right)dy,\label{eq:3.21}
\end{equation}
where
\begin{equation}
\varphi(y,\chi)=2\sum_{n=1}^{\infty}n^{\kappa}\chi(n)\exp\left(-\frac{n^{2}\pi}{q}e^{2y}+\left(\kappa+\frac{1}{2}\right)y\right),\label{eq:3.22}
\end{equation}
then
\begin{equation}
\xi\left(\frac{1}{2}+is,\chi\right)=\sum_{n=0}^{\infty}i^{n}a_{n}(\chi)s^{n},\quad a_{n}(\chi)=\int_{-\infty}^{\infty}y^{n}\varphi\left(y,\chi\right)dy.\label{eq:3.23}
\end{equation}
It is known that the fast decreasing smooth function $\varphi(y,\chi)$
satisfies the functional equation \cite{Davenport}
\begin{equation}
\varphi(y,\chi)=\frac{i^{\kappa}\sqrt{q}}{\tau\left(\overline{\chi}\right)}\varphi(-y;\overline{\chi}),\quad y\in\mathbb{R},\label{eq:3.24}
\end{equation}
then for all $n\in\mathbb{N}_{0}$,
\begin{equation}
\begin{aligned} & a_{n}(\overline{\chi})=\int_{-\infty}^{\infty}y^{n}\varphi\left(y,\overline{\chi}\right)dy=(-1)^{n}\int_{-\infty}^{\infty}y^{n}\varphi\left(-y,\overline{\chi}\right)dy\\
 & =\frac{(-1)^{n}\tau\left(\overline{\chi}\right)}{i^{\kappa}\sqrt{q}}\int_{-\infty}^{\infty}y^{n}\varphi\left(y,\chi\right)dy=\frac{(-1)^{n}\tau\left(\overline{\chi}\right)}{i^{\kappa}\sqrt{q}}a_{n}(\chi).
\end{aligned}
\label{eq:3.25}
\end{equation}
Let
\begin{equation}
f(s,\chi)=s^{-2\mu}G\left(\frac{1}{2}+is,\chi\right)=s^{-2\mu}\xi\left(\frac{1}{2}+is,\chi\right)\cdot\xi\left(\frac{1}{2}+is,\overline{\chi}\right),\label{eq:3.26}
\end{equation}
where $\mu\in\mathbb{N}_{0}$ is the least nonnegative integer such
that $a_{\mu}(\chi)\neq0$, then the even entire function $f(s,\chi)$
has the series expansion
\begin{equation}
f(s,\chi)=f(-s,\chi)=\sum_{n=0}^{\infty}(-1)^{n}b_{n}(\chi)s^{2n},\label{eq:3.27}
\end{equation}
where $b_{0}(\chi)=\frac{(-1)^{\mu}\tau\left(\overline{\chi}\right)}{i^{\kappa}\sqrt{q}}a_{\mu}^{2}(\chi)$
and for $n\in\mathbb{N}$,
\begin{equation}
\begin{aligned} & b_{n}(\chi)=\sum_{j=0}^{2n}a_{j+\mu}(\chi)a_{2n-j+\mu}(\overline{\chi})\\
 & =\frac{(-1)^{\mu}\tau\left(\overline{\chi}\right)}{i^{\kappa}\sqrt{q}}\sum_{j=0}^{2n}(-1)^{j}a_{j+\mu}(\chi)a_{2n-j+\mu}(\chi).
\end{aligned}
\label{eq:3.28}
\end{equation}
It is well-known that $\xi(s,\chi)$ is an order $1$ entire function
with infinitely many zeros, all of them are in the horizontal strip
$t\in(-1/2,1/2)$, \cite{Davenport}. Then $f(s,\chi)$ is an order
$1$ even entire function with infinitely many zeros, all of them
are in the horizontal strip $t\in(-1/2,1/2)$. Clearly, all the zeros
of $\xi(s,\chi)$ on the critical line $\sigma=\frac{1}{2}$ if and
only if all the zeros of $f(s,\chi)$ are real. Therefore, the generalized
Riemann hypothesis for $L\left(s,\chi\right)$ is equivalent to that
all the zeros $\left\{ \pm z_{n}(\chi)\right\} _{n\in\mathbb{N}}$
of $f(s,\chi)$ are real where $\Re(z_{n}^{2})>0$. 

Let
\begin{equation}
\lambda_{n}(\chi)=z_{n}^{2}(\chi),\quad n\in\mathbb{N},\label{eq:3.29}
\end{equation}
 then it is know that for any $\epsilon>0$, \cite{Davenport,Edwards}
\begin{equation}
\sum_{n=1}^{\infty}\frac{1}{\left|\lambda_{n}(\chi)\right|^{1/2+\epsilon}}<\infty.\label{eq:3.30}
\end{equation}
Let 
\begin{equation}
\Theta_{\Xi}(t,\chi)=\sum_{n=1}^{\infty}e^{-\lambda_{n}^{2}(\chi)t},\quad\forall t>0,\label{eq:3.31}
\end{equation}
then for all $k\in\mathbb{N}_{0}$ by Lemma \ref{lem:2} we have
\begin{equation}
\Theta_{\Xi}^{(k)}(t,\chi)=\mathcal{O}\left(t^{-1/2-\epsilon-k}\right),\quad t\downarrow0.\label{eq:3.32}
\end{equation}

Similar to Corollary \ref{cor:5} we apply Theorem \ref{thm:main}
to $f(s,\chi)$ we obtain the following corollary:
\begin{cor}
\label{cor:4-1}The generalized Riemann hypothesis for a primitive
Dirichlet character $\chi$ is true if and only if there exists a
nonnegative integer $m$ all the functions
\begin{equation}
\left(-\frac{1}{x}\right)^{m}\frac{d^{k}}{dx^{k}}\left(x^{k+m}\frac{d^{m}}{dx^{m}}\left(\frac{f'\left(\sqrt{x},\chi\right)}{\sqrt{x}f\left(\sqrt{x},\chi\right)}\right)\right),\quad k\in\mathbb{N}_{0}\label{eq:3.33}
\end{equation}
are completely monotonic on $(0,\infty)$ where $f(s,\chi)$ is defined
by (\ref{eq:2.25}).
\end{cor}

\end{document}